\documentclass[a4paper,reqno]{amsart}
\usepackage{amssymb}

\theoremstyle{plain}
\newtheorem{thm}{Theorem}
\newtheorem{lem}{Lemma}
\newtheorem{prop}{Proposition}
\newtheorem{cor}{Corollary}
\theoremstyle{definition}
\newtheorem{defn}{Definition}
\newtheorem{exmp}{Example}

\newcommand\Comp{\mathop{\fam 0 Comp}\nolimits}
\newcommand\idd{\mathop{\fam 0 id}\nolimits}
\newcommand\ann{\mathop{\fam 0 ann}\nolimits}
\newcommand\Spann{\mathop{\fam 0 Span}\nolimits}
\newcommand\oo[1]{\mathrel{{}_{#1}}}
\newcommand\Hom{\mathop {\fam 0 Hom}\nolimits}
\newcommand\End{\mathop {\fam 0 End}\nolimits}
\newcommand\Cend{\mathop {\fam 0 Cend}\nolimits}
\newcommand\Diff{\mathop {\fam 0 Diff}\nolimits}
\def\Vecc{\mathrm{Vec}}
\def\Alg{\mathrm{Alg}}

\title[On irreducible algebras of conformal
endomorphisms]{On irreducible algebras of conformal
endomorphisms over a linear algebraic group}

\author{P.~S.~Kolesnikov}
\address{Sobolev Institute of Mathematics, Novosibirsk, Russia}
\email{pavelsk@math.nsc.ru}

\thanks{Partially supported by RFBR (project 05--01--00230).
The author gratefully acknowledges the support of the Pierre Deligne fund
based on his 2004 Balzan prize in mathematics.}


\begin{document}

\begin{abstract}
We study the algebra of conformal endomorphisms
$\Cend^{G,G}_n$ of a finitely generated free module
$M_n$ over the coordinate Hopf algebra $H$ of a linear algebraic
group~$G$.
It is shown that a conformal subalgebra of
$\Cend_n$ acting irreducibly on $M_n$
generates an essential left ideal of $\Cend^{G,G}_n$
if enriched with operators of multiplication on elements of $H$.
In particular, we describe such subalgebras for the case when $G$ is finite.
\end{abstract}

\maketitle

\subsection*{Introduction}

The notion of a conformal algebra was introduced in
\cite{K1}
as a tool for investigation of vertex algebras \cite{Bor, FLM}.
From the formal point of view, a conformal algebra is a linear space
$C$
over a field
$\Bbbk$ ($\mathrm{char}\,\Bbbk =0$)
endowed with a linear operator
$T:C\to C$ and with a family of bilinear operations
$(\cdot\oo{n}\cdot)$, $n\in \mathbb Z_+$
(where $\mathbb Z_+$ stands for the set of non-negative integers),
satisfying the following axioms:
\begin{itemize}
\item[(C1)]
 for every $a,b\in C$ there exists $N\in \mathbb Z_+$
such that $(a\oo{n} b)=0$ for all $n\ge N$;
\item[(C2)]
 $(Ta\oo{n} b) = T(a\oo{n} b) - (a\oo{n} Tb) =
 \begin{cases} -n(a\oo{n-1}b), & n>0 \\ 0,& n=0\end{cases}$
\end{itemize}
One of the most natural examples of a conformal algebra
is the Weyl conformal algebra
$\mathcal W = \Bbbk[T]^{\otimes 2}\simeq \Bbbk[T,v]$,
where the operations
$(\cdot\oo{n}\cdot)$ are defined as follows:
$$
T^{(r)} f(v)\oo{n} T^{(s)} h(v) = \sum\limits_{t\ge 0}
   (-1)^r\binom{n}{r}\binom{n-r}{t} T^{(s-t)} f(v)\partial^{n-r-t}h(v),
$$
$r,s\in \mathbb Z_+$, $f,h\in \Bbbk[v]$.
Hereinafter, $T^{(m)}=T^m/m!$ (it is suitable to suppose $T^{(m)}=0$
for $m<0$),
$\partial = d/dv$ is the ordinary derivation with respect to
the variable~$v$.

The collection of operators
$W =\{(a\oo{n}\cdot)\in \End \mathcal W : a\in \mathcal W,\, n\in \mathbb Z_+\}$
is a subalgebra of the algebra of linear transformations  of the space
$\mathcal W$, moreover,
$W\simeq A_1=\Bbbk\langle x ,d\mid dx-xd=1\rangle$,
i.e., $W$ is isomorphic to the first Weyl algebra
(this is a reason for the name of the conformal algebra $\mathcal W$).

The canonical representation of
 $A_1$ on the space of polynomials gives rise to the natural action
of the conformal algebra $\mathcal W$ on $\Bbbk[T]$.
Namely, the family of operations
$(\cdot\oo{n}\cdot): \mathcal W\otimes \Bbbk[T]\to \Bbbk[T]$,
$n\in \mathbb Z_+$,
given by
$$
 T^{(r)}f(v)\oo{n} T^{(s)} = (-1)^r\binom{n}{r}T^{(s+r-n)}f(T),
  \quad f\in \Bbbk[v],\ r,s\ge 0,
$$
satisfies the conditions similar to
(C1) and~(C2).

For an integer $N\ge 1$,
the set
$\mathbb M_n(\mathcal W)$
of all matrices of size $N$ over $\mathcal W$
is also a conformal algebra
(denoted by $\Cend_N$ \cite{K1, K3})
which acts on the space of columns
$M_N=\Bbbk[T]\otimes \Bbbk^N$
by the ordinary matrix multiplication rule.

In \cite{K1}, V.~Kac stated the problem to
describe {\em irreducible\/} conformal subalgebras
 $C\subseteq \Cend_N$, i.e.,
such that there are no non-trivial
$\Bbbk[T]$-submodules of $M_N$ invariant with respect to the
operators $(a\oo{n}\cdot)$,
$a\in C$, $n\in \mathbb Z_+$.
In \cite{BKL}, such a description was obtained for
$N=1$
and also for the case when
 $C$ is a finitely generated module over $\Bbbk[T]$.
In the same paper, a conjecture on the structure of irreducible subalgebras
of $\Cend_n$ was stated. The conjecture was proved in~\cite{Ko1}.

The classification of irreducible conformal subalgebras of
$\Cend_N$
leads to a description of a class of ``good''
subalgebras of the matrix Weyl algebra
$\mathbb M_N(A_1)$
acting irreducibly on
$M_N$
(see \cite{Ko4}).
Note that the problem of description of all such subalgebras of
$\mathbb M_N(A_1)$ remains open even for $N=1$.

In the present paper, we consider the notion of a conformal algebra
over a linear algebraic group $G$. The class of such algebras includes
ordinary algebras over a field
(for $G=\{e\}$)
and conformal algebras (for $G=\mathbb A^1\simeq (\Bbbk, +)$).
On the other hand, the notion of a conformal algebra over
$G$ is equivalent to the notion of a pseudo-algebra \cite{BDK} over $H$,
where $H=\Bbbk[G]$ is the Hopf algebra of regular functions
(coordinate algebra) on~$G$.

The analogue of
$\Cend_N$ in the class of conformal algebras over $G$
can be considered as a collection of transformation rules
(satisfying certain conditions)
of the space of vector-valued regular functions
$u: G\to \Bbbk^N$
by means of elements of the group~$G$.
For example, the left-shift transformation
$L: \gamma \mapsto L_\gamma $,
where $(L_\gamma u)(x)=u(\gamma x)$, $\gamma, x\in G$,
belongs to~$\Cend_N$.

We will prove a generalization of the main result of
\cite{Ko1}
for the case of an arbitrary linear algebraic group $G$.
For $G=\{e\}$ this would be the classical Burnside Theorem,
and if
$G=\mathbb A^1$ then it turns out to be crucial point of proof
of the conjecture from \cite{BKL}.
We will apply the result obtained in the case of ``intermediate'' complexity
to describe irreducible conformal subalgebras over a finite group.

\subsection{Multicategories and operads}
In this section, we state some notions related to operads
 \cite{GK} and multicategories
(also known as pseudo-tensor categories \cite{BD}).

An ordered $n$-tuple of integers
$\pi=(m_1,\dots ,m_n)$, $m_i\ge 1$,
is said to be an $n$-{\em partition\/} of $m$
if  $m_1+\dots + m_n = m$.
The set of all such partitions is denoted by
$\Pi(m,n)$.
Each partition
$\pi\in \Pi(m,n)$
defines a bijective correspondence between the sets
$\{1,\dots, m\}$
and
$\{(i,j)\mid i=1,\dots, n, \, j=1,\dots, m_i\}$,
namely,
\[
 (i,j) \leftrightarrow (i,j)^\pi := m_1+\dots + m_{i-1} + j.
\]
For two partitions
\[
\tau=(p_1,\dots ,p_m)\in \Pi(p,m),
\quad
\pi = (m_1,\dots ,m_n)\in \Pi(m,n),
\]
define
$\tau\pi \in \Pi(p,n)$ in the following way:
\[
\tau\pi = (p_1+\dots+p_{m_1}, p_{m_1+1}+\dots + p_{m_1+m_2},
  \dots, p_{m-m_n+1} +\dots+p_m).
\]
If
$\tau\pi = (q_1,\dots, q_n)$
then for every
$i=1,\dots, n$
let
$\tau\pi_i$
denotes the subpartition
$(p_{m_1+\dots+m_{i-1}+1}, \dots, p_{m_1+\dots+m_i}) \in \Pi(q_i,m_i)$.

Suppose $\mathcal A$ is a class of objects
such that for every integer
$n\ge 1$
and for every family
 $A_1,\dots, A_n, A\in \mathcal A$
there exists a linear space
$P^{\mathcal A}_n(A_1,\dots, A_n; A)=P_n(\{A_i\}; A)$
over a fixed field $\Bbbk $.

Also, suppose that for every
$A_1,\dots ,A_m\in \mathcal A$,
$B_1,\dots ,B_n\in \mathcal A$,
$C\in \mathcal A$
and for every
$n$-partition
$\pi =(m_1,\dots ,m_n)$ of~$m$
there exists a linear map
\begin{equation}\label{CH2-Comp}
\Comp^\pi :
P_n(\{B_i\}; C)
 \otimes  \bigotimes\limits_{i=1}^n
P_{m_i}(\{A_{(i,j)^\pi }\}; B_i)
\to P_m(A_1,\dots A_m; C).
\end{equation}
To shorten the notation, we will denote
$ \Comp^\pi (\varphi ,\psi_1,\dots, \psi_n)$
by
$\Comp^\pi (\varphi ,\{\psi_i\})$
or by
$\varphi (\psi_1,\dots ,\psi _n)$,
if the structure of a partition
$\pi $
is clear.

The elements of the spaces
$P_n(\{A_i\}; B)$ are called
{\em multimorphisms} (or $n$-{\em mor\-phisms}),
the family of maps $\Comp^\pi $,
$\pi \in \Pi(m,n)$, $m,n\ge 1$,
is called a {\em composition rule}.

A class $\mathcal A$ endowed with spaces of multimorphisms
and a composition rule is called a
{\em multicategory\/} if the following axioms hold.

(A1)
The composition rule is associative.
The latter means that, given a collection of objects
$A_h, B_j, C_i \in \mathcal A$
($h=1,\dots ,p$, $j=1,\dots ,m$, $i=1,\dots ,n$),
two partitions
$\tau=(p_1,\dots ,p_m)\in \Pi(p,m)$,
$\pi = (m_1,\dots ,m_n)\in \Pi(m,n)$,
an object
 $D\in \mathcal A$,
an $n$-morphism
$\varphi \in P_n(\{C_i\}; D)$,
and two collections of multimorphisms
\[
 \psi_j\in P_{p_j}(\{A_{(j,t)^\tau }\}; B_j),\ j=1,\dots, m,
\quad
\chi_i \in P_{m_i}(\{B_{(i,t)^\pi }\}; C_i),\ i=1,\dots, n,
\]
we have
\begin{equation}   \nonumber
\Comp^\tau \big(\Comp^\pi (\varphi ,\{\chi_i\}),
 \{\psi_j\} \big)
=
\Comp^{\pi\tau}\big(\varphi ,
\big\{\Comp^{\tau_i} (\chi_i, \{\psi_{(i,t)}\} )
    \big\} \big),
\end{equation}
where $\tau_i = (p_{i1},\dots ,p_{im_i})$ are subpartitions of~$\tau$.

(A2)
For every
$A\in \mathcal A$
there exists an ``identity'' 1-morphism
$\idd_A \in P_1(A; A)$
such that
\begin{equation}  \nonumber
f(\idd_{A_1},\dots , \idd_{A_n}) =
\idd_A(f) = f
\end{equation}
for all $f\in P_n(\{A_i\}; A)$, $n\ge 1$.

Let
$\mathcal A$ and $\mathcal B$ be two multicategories.
A {\em functor\/} from $\mathcal A$ to $\mathcal B$
is a rule
$F$ which maps an object
$A\in \mathcal A$ to $F(A)\in \mathcal B$ in such a way that
for every
$\varphi \in P^{\mathcal A}_n(A_1,\dots, A_n; A)$
there exists
$F(\varphi )\in P^{\mathcal B}_n(F(A_1),\dots, F(A_n); F(A))$,
where
\[
\begin{gathered}
F ( \Comp^\pi (\varphi ,\{\psi_i \} ))
= \Comp^\pi (F(\varphi ), \{F(\psi _i)\} );
\\
F(\idd_A) = \idd_{F(A)},
\end{gathered}
\]
and the map
$\varphi\mapsto F(\varphi )$ is linear.

One of the most natural examples of a multicategory is  the class
 $\Vecc_{\Bbbk}$
of linear spaces over a field~$\Bbbk $
with respect to
\[
P_n^{\Vecc_\Bbbk}(A_1,\dots ,A_n;A) = \Hom (A_1\otimes \dots \otimes A_n, A),
\]
where the rule
$\Comp $ is the ordinary composition of multilinear maps.

The multicategory
$\Vecc_\Bbbk $ can be considered as a particular case
of multicategory $\mathcal M^*(H) $ of left unital modules
over a (coassociative) bialgebra $H$ (see \cite{BDK} for details).

Let us consider one more example of a multicategory.
For every integer  $n\ge 1$ denote by $\Alg(n)$
the linear space spanned by all binary trees with
$n$ leaves. Such a tree can be naturally identified with a bracketing
on the word
$x_1\dots x_n$ over the alphabet $X=\{x_1,x_2,\dots \}$.
The composition of binary trees can be expressed as
\eqref{CH2-Comp} by means of the composition of words as follows:
\begin{multline}\label{CH2-comp_Alg}
 \Comp^\pi (u, v_1,\dots, v_n) \\
=
u(v_1(x_{(1,1)^\pi}, \dots , x_{(1,m_1)^\pi}),
\dots ,
v_n(x_{(n,1)^\pi}, \dots , x_{(n,m_n)^\pi})
)
\end{multline}
where
$\pi \in \Pi(m,n)$,
$u=u(x_1,\dots, x_n)\in \Alg(n)$,
$v_i=v_i(x_1,\dots, x_{m_i})\in \Alg(m_i)$,
$i=1,\dots, n$.

The class $\Alg $ which consists of a single object
endowed with multimorphisms
$P_n = \Alg(n)$ and the composition rule
\eqref{CH2-comp_Alg}
is a multicategory.
A multicategory with a single object is known as
{\em operad}.

\begin{defn}[see, e.g., \cite{GK}]\label{def:alg}
An {\em algebra\/} in a multicategory $\mathcal A$
is a functor $F$ from the operad $\Alg $ to $\mathcal A$.
\end{defn}

Every functor $F$ from $\Alg $ to $\mathcal A$
is completely defined by an object
$A \in \mathcal A$ and by a 2-morphism
$\mu =F(x_1x_2)\in P_2^{\mathcal A}(A,A;A)$.
If $\mathcal A=\Vecc_\Bbbk$ then
$\mu $ is an ordinary product; if $\mathcal A=\mathcal M^*(H)$
then  $\mu $ is called a {\em pseudo-product}.
Algebras in the multicategory $\mathcal M^*(H)$ are known as {\em pseudo-algebras\/}
over $H$,  or $H$-{\em pseudo-algebras} \cite{BDK}.

An algebra $F$ in the sense of Definition \ref{def:alg}
is said to be {\em associative\/} if
\[
F(x_1(x_2x_3)) = F((x_1x_2)x_3) .
\]
In order to define what is, for example, a commutative or Lie algebra,
one needs the notion of a symmetric multicategory.
In the present paper we consider associative algebras only, so
we do not need a symmetric structure.

\subsection{Conformal algebras over a linear algebraic group}
Consider a linear algebraic group
$G$
and the corresponding coordinate Hopf algebra
$H=\Bbbk[G]$
with coproduct
$\Delta $, counit $\varepsilon $, and antipode~$S$.
We will use the short Sweedler notation as follows:
\[
\begin{gathered}
\Delta(h)=h_{(1)}\otimes h_{(2)},
\quad
(\Delta\otimes \idd_H)\Delta(h)=
(\idd_H\otimes \Delta)\Delta(h)= h_{(1)}\otimes h_{(2)}\otimes h_{(3)},
\\
(S\otimes \idd_H)\Delta(h) = h_{(-1)}\otimes h_{(2)},
\quad
(\idd_H\otimes S)\Delta(h) = h_{(1)}\otimes h_{(-2)},
\end{gathered}
\]
and so on.

Denote by $L_g$, $g\in G$,
the operator of left shift on
$H$, i.e.,
$L_gh = h_{(1)}(g)h_{(2)}$, $h\in H$.
It is clear that
$L_{g_1}L_{g_2} = L_{g_2g_1}$.

Introduce the following multicategory structure on the class
of all left unital modules over the algebra $H$.
Let $M_1,\dots, M_n, M$ be such modules,
and let $ P_n(M_1,\dots, M_n; M)$ stands for the space of all functions
\[
a: G^{n-1}\to \Hom_\Bbbk( M_1\otimes \dots  \otimes M_n, M)
\]
such that
\begin{itemize}
\item
     for every $u_i\in M_i$, $i=1,\dots, n$, the map
$G^{n-1} \to M$ defined by
$x\mapsto a(x)(u_1,\dots, u_n)$
is a regular
$M$-valued function on $G^{n-1}$;
\item
for every $u_i\in M_i$, $f_i\in H$, $i=1,\dots, n$,
$g_1,\dots, g_{n-1}\in G$
we have
\[
\begin{split}
&a(g_1,\dots, g_{n-1}) (f_1u_1, \dots, f_nu_n) \\
&\quad
=f_1\big(g_1^{-1}\big)\dots f_{n-1}\big(g_{n-1}^{-1}\big)
(L_{g_1}\dots L_{g_{n-1}}f_n ) a(g_1,\dots, g_{n-1})(u_1,\dots, u_n).
\end{split}
\]
\end{itemize}

Define the following composition rule on the spaces mentioned above.
Suppose we are given
$\pi =(m_1,\dots, m_n)\in \Pi(m,n)$,
$\psi_i\in P_{m_i}(\{N_{(i,j)^\pi}\}; M_i)$,
$i=1,\dots, n$,
$\varphi \in P_n(\{M_i\}; M)$.
Consider a family of
$g_j\in G$, $j=1,\dots, m-1$,
and put
$\gamma_i = g_{(i,m_i)^\pi}\dots g_{(i,1)^\pi}\in G$
for
$i=1,\dots, n-1$,
$\bar g_i =(g_{(i,1)^\pi}\dots g_{(i,m_i-1)^\pi})\in G^{m_i-1}$
for
$i=1,\dots, n$
(if $m_i=1$ then $\bar g_i$ is void),
$\bar \gamma =(\gamma _1,\dots, \gamma_{n-1})\in G^{n-1}$.
Now, define
\begin{equation}
\Comp^\pi (\varphi, \psi_1,\dots, \psi_n)(g_1,\dots , g_{m-1})
=
\Comp^\pi (\varphi(\bar \gamma), \psi_1(\bar g_1),\dots, \psi_n(\bar g_n)),
                            \label{eq:G-comp}
\end{equation}
where $\Comp $ in the right-hand side
means the composition rule of the multicategory
$\Vecc_\Bbbk $.

Let us denote the multicategory constructed above by
$\mathcal M^*(H)$,
since it is not difficult to note that this is a particular
case of the multicategory from \cite{BDK}.

\begin{defn}\label{defn:Gconf}
A {\em conformal algebra\/} over a linear algebraic group
$G$ is an algebra in the multicategory
$\mathcal M^*(H)$, $H=\Bbbk[G]$.
\end{defn}

In the language of ``ordinary'' algebraic operations a conformal algebra
over $G$ can be defined as a left $H$-module $C$
endowed with a family of
$\Bbbk $-linear operations
$(\cdot\oo{\gamma}\cdot): C\otimes C \to C$,
$\gamma \in G$, such that  for every
$a,b\in C$, $h\in H$, $g\in G$ we have:
\begin{itemize}
\item[(G1)]  the map
$(a\oo{x}b): \gamma \mapsto (a\oo{\gamma} b)$
is a regular $C$-valued function on~$G$;
\item[(G2)]
$(ha\oo{g}b) = h(g^{-1})(a\oo{g}b)$;
\item[(G3)]
$(a\oo{g}hb)=L_gh(a\oo{g} b)$.
\end{itemize}

A conformal algebra  over a trivial group
$G=\{e\}$ is just an ordinary algebra over the field $\Bbbk $;
if $G=\mathbb A^1$ (the affine line)
then we get the definition of a conformal algebra
\cite{K1} in terms of $\lambda$-brackets;
in the case $G=\mathrm{GL}_1(\Bbbk )$ we obtain
the notion of a $\mathbb Z$-conformal algebra \cite{GKK}.

The associativity property
of a conformal algebra can be easily expressed in terms of
the operations
$(\cdot \oo{g}\cdot )$ by means of computation the images of
$x_1(x_2x_3) = \Comp^{(1,2)}(x_1x_2, x_1, x_1x_2)$
and $(x_1x_2)x_3 = \Comp ^{(2,1)}(x_1x_2, x_1x_2, x_1)$.
According to
\eqref{eq:G-comp}, a conformal algebra $C$ over $G$ is associative
if and only if
\begin{equation}                    \label{eq:conf-assoc}
a\oo{g}(b\oo{\gamma } c) = (a\oo{g} b)\oo{\gamma g} c
\end{equation}
for all $a,b,c\in C$, $g,\gamma \in G$.

\begin{exmp}
Let $A$ be an $H$-comodule algebra, i.e.,  a (not necessarily
associative) algebra equipped by a coaction map
$\Delta_A : A\to H\otimes A$
such that
\begin{itemize}
\item $\Delta_A $ is a homomorphism of algebras;
\item $(\Delta\otimes \idd_A)\Delta_A = (\idd_H\otimes \Delta_A)\Delta_A$.
\end{itemize}
We will use the Sweedler notation for
$\Delta_A$ as well as for~$\Delta $.
Then the free
$H$-module  $C=H\otimes A$ under the operations
\begin{equation}\label{eqGprod-tmp}
(h\otimes a)\oo{\gamma}(f\otimes b) = h(\gamma^{-1})b_{(1)}(\gamma )L_\gamma f
 \otimes ab_{(2)}, \quad \gamma\in G,\ f,h\in H,\ a,b\in A,
\end{equation}
is a conformal algebra over $G$. Moreover, if $A$ is associative
then  $C$ satisfies \eqref{eq:conf-assoc}.
Following \cite{Re1},
denote the conformal algebra $C$ obtained by
$\Diff(A,\Delta_A)$.

Note that an arbitrary algebra $A$ is an $H$-comodule algebra
with respect to the coaction
$\Delta^0_A (a) =1\otimes a$, $a\in A$.
The conformal algebra
$\Diff(A,\Delta_A^0)$ is called the loop algebra or current algebra over $A$,
denoted by $\mathrm{Cur}^H A$.
\end{exmp}

\subsection{Conformal linear maps}

Let
 $M$ and $N$
be linear spaces over the field~$\Bbbk $.
The set $\Hom(M,N)$ is a topological linear space
with respect to the finite topology
(see, e.g., \cite{Jac}).
In that sense, a sequence
 $\{\alpha _k\}_{k\ge 0}\subset \Hom(M,N)$
converges to a map
 $\alpha \in \Hom (M,N)$
if and only if for every finite number of elements
$u_1,\dots, u_n\in M$
there exists a natural $m$ such that
$\alpha _k(u_i)=\alpha (u_i)$, $i=1,\dots ,n$,
for all $k\ge m$.

\begin{defn}\label{defn1}
A map
$a: G \to \Hom (M,N)$
is said to be {\em locally regular\/}
if for every
$u\in M$ the map
$z\mapsto a(z)u$, $z\in G$,
is a regular $N$-valued function on~$G$.
\end{defn}

Let $\{h_i\}_{i\in I}$ be a linear basis of $H$ over $\Bbbk $.
Consider the dual space $H^*$ endowed with a topology defined
by the basic neighborhoods of zero of the form
\[
  X(i_1,\dots ,i_n) = (\Spann_{\Bbbk} \{h_{i_1},\dots ,h_{i_n}\})^\perp
  \subseteq H^*,\quad  i_k\in I,\ k=1,\dots ,n,\ n\ge 0.
\]
It is clear that, up to equivalence, the topology does not depend on the
choice of basis in~$H$.

Suppose $a:G\to \Hom (M,N)$ is a locally regular map.
Then for every
$u\in M$ the function $a(x)u: G\to N$ can be presented as
$a(x) u = \sum\limits_{i\in I} h_i\otimes v_i\in H\otimes N$,
where $a(z)u = \sum\limits_{i\in I}h_i(z)v_i$.
For any $\xi \in H^*$ define
$\tilde a(\xi )\in \Hom(M,N)$ in the following way:
\[
  \tilde a(\xi ): u\mapsto (\langle \xi ,\cdot\rangle \otimes \idd)(a(x) u),
\quad u\in M.
\]
Let us denote by $\tilde a$ the linear map $H^*\to \Hom (M,N)$
which maps $\xi $ to $\tilde a(\xi )$.

\begin{lem}\label{loc-cont}
{\rm 1.}
If $a$ is locally regular then
$\tilde a $ is continuous with respect to the finite topology on
$\Hom(M,N)$.

{\rm 2.} For any continuous linear map
$\alpha :H^*\to \Hom(M,N)$
there exists a locally regular
$a:G\to \Hom(M,N)$
such that
$\tilde a = \alpha $.
\end{lem}

\begin{proof}
1. Due to the local regularity of $a$,
for every
 $u_1,\dots ,u_n\in M$
there exist only a finite number of basic
$h_{i_1},\dots ,h_{i_m}\in H$
that appear in the presentations of
$a(x)u_j$, $j=1,\dots ,n$.
Therefore, if $\xi \in X(i_1,\dots ,i_m)$ then
$\tilde a(\xi )u_j =0$ for all $j=1,\dots, n$,
i.e.,
$\tilde a(\xi )\to 0$ as $\xi \to 0$
with respect to the finite topology  on $\Hom (M,N)$.

2. For each $i\in I$ denote by $\xi _i$ the element of $H^*$
defied by the rule
$\langle \xi _i, h_j\rangle =\delta _{i,j}$,
$i,j\in I$.
For a fixed $u\in M$ the set $\{i\in I\mid \alpha(\xi_i)u\ne 0\}$
is finite. Hence, we may consider the map
$a(g)\in \Hom (M,N)$, $g\in G$,
given as follows:
\[
  a(g) u = \sum\limits_{i\in I}h_i(g)\alpha (\xi _i)u.
\]
The function $a(x)u: g\mapsto a(g)u$
is regular, and for the locally regular
$a: G\to \Hom (M,N)$ we have $\tilde a=\alpha $.
\end{proof}

Let
$G_1$, $G_2$ be two linear algebraic groups,
$H_i=\Bbbk[G_i]$, $i=1,2$, be the corresponding Hopf algebras,
then $\Bbbk[G_1\times G_2]= H_1\otimes H_2$.

Suppose we are given a locally regular map
$a : G_1\times G_2\to \Hom(M,N)$.
For every
$\xi \in H_1^*$, $u\in M$
define
\[
\tilde a(\xi ,y)u =
(\langle \xi ,\cdot \rangle \otimes \idd_{H_2}\otimes \idd_N)(a(x,y)u),
\]
where $(x,y)$ is a variable that ranges over $G_1\times G_2$.
We obtain a function
\[
  \tilde a: H_1^*\times G_2\to \Hom(M,N),
\quad
 (\xi , z)\mapsto \tilde a(\xi ,z), \quad \xi \in H_1^*, \ z\in G_2.
\]

\begin{lem}\label{lem:crossprod}
For every $z\in G_2 $ the map
$\xi \mapsto\tilde a(\xi , z)$
is continuous.
For every $\xi \in H_1^*$ the map
$z\mapsto \tilde a(\xi , z)$ is locally regular.
\end{lem}

\begin{proof}
If $u\in M$, $\{h_i\}$ and $\{f_j\}$ are bases of $H_1$ and $H_2$,
respectively, then
$a(x,y)u = \sum\limits_{i,j} h_i(x)\otimes f_j(y)\otimes v_{ij}$,
where the sum is finite.
Both statements obviously follow from this presentation.
\end{proof}

Let $V$ be a $G$-set, i.e., a Zariski closed subset of
an affine space endowed with a continuous action of the group~$G$.
The algebra $A$ of regular functions on $V$
is an $H$-comodule algebra with the coaction $\Delta_A : A \to H\otimes A$
dual to the action of $G$ on~$V$. There exists a natural (right)
representation of $G$ on $A$ by the left-shift automorphisms,
namely,
$L: g \mapsto L_g$, $g\in G$,
 where
$(L_gf)(v) = f(gv)$,  $f\in A$, $v\in V$.

The representation $L$, in particular, has the following properties:
\begin{itemize}
\item
     for any $f\in A$ the map $g\mapsto L_gf$ is an
    $A$-valued regular map on~$G$;
\item
     $L_g(f_1f_2) = L_g(f_1) L_g(f_2)$ for all $f_1,f_2\in A$, $g\in G$.
\end{itemize}

Consider the space of operators that have the same properties as $L$.

\begin{defn}
A {\em conformal linear transformation} ({\em conformal endomorphism})
of an $A$-module $M$ is a locally regular map
$a : G \to \End M$
such that
\begin{equation}\label{TinvCend}
a(g)(fu) = L_gf( a(g)u),\quad  u\in M,\ f\in A,\ g\in G.
\end{equation}
The property \eqref{TinvCend}
is called {\em translation invariance\/}
or T-invariance of the map~$a$.

Denote by $\Cend^{G,V} M$
the space of all conformal endomorphisms
of an $A$-module $M$.
\end{defn}

Introduce the following structure of an
$H$-module on $C=\Cend^{G,V} M$:
\begin{equation}\label{eq:H-mod}
  (fa)(g) = f(g^{-1})a(g),\quad f\in H,\ a\in C,\ g\in G.
\end{equation}
Also, consider the operations
 $(\cdot \oo{g} \cdot )$, $g\in G$,
given by
\begin{equation}\label{eq:G-prod}
 (a \oo{g} b)(z) = a(g)b(zg^{-1}), \quad a,b\in C,\ g,z\in G
\end{equation}
(it is easy to check that $(a\oo{g} b)$
is a locally regular map satisfying \eqref{TinvCend}).

\begin{prop}\label{propCend}
The space
$\Cend^{G,V} M$, where $M$ is a finitely generated $A$-module,
is an associative conformal algebra over $G$
with respect to the operations
\eqref{eq:H-mod}, \eqref{eq:G-prod}.
\end{prop}

\begin{proof}
Let us check the conditions (G1)--(G3) for the $H$-module $C=\Cend^{G,V} M$.
It follows from \eqref{TinvCend} that
\eqref{eq:H-mod}, \eqref{eq:G-prod}
satisfy (G2) and~(G3).

To check (G1), consider $u\in M$, $a,b\in C$,
$z,g\in G$.
Relation \eqref{eq:G-prod} implies that
the map $c: (g,z)\mapsto (a\oo{g} b)(z)\in \End M$
is a locally regular function on
$G\times G$. Hence, by Lemma~\ref{lem:crossprod},
for every $\xi \in H^*$ the map
$\tilde c(\xi , x ): G\to \End M$, $z\mapsto \tilde c(\xi ,z)$,
locally regular.
It is also clear that $\tilde c(\xi, x)$ is T-invariant.
Therefore, $\tilde c(\xi , x)\in C $.

Suppose $e_1, \dots , e_n\in M$ is a set of generators
of the $H$-module~$M$.
There exist finite collections of elements
$f^1_{ik}, h^1_{ik},
f^2_{ik}, h^2_{ik}\in H$,
$i,k=1,\dots ,n$,
such that
\[
 b(x)e_i = \sum\limits_{k=1}^{n} f^1_{ik}(x)\otimes f^2_{ik}e_k, \quad
 a(x)e_i = \sum\limits_{k=1}^{n} h^1_{ik}(x)\otimes h^2_{ik}e_k.
\]
Then
\begin{multline}   \nonumber
c(x,y)e_i = (a\oo{x} b)(y)e_i = a(x)b(yx^{-1})e_i
=\sum\limits_{k,l=1}^{n} f^1_{ik}(yx^{-1})h^1_{kl}(x)(L_xf^2_{ik})h^2_{kl}e_l
\\
=\sum\limits_{k,l=1}^{n} f^1_{ik(-2)}f^2_{ik(1)}h_{kl}\otimes f^1_{ik(1)}
\otimes f^2_{ik(2)}h^2_{kl}e_l.
\end{multline}
If $\xi \in H^*$ is sufficiently close to zero then
$\tilde c(\xi ,z)e_i=0$ for all
$z\in G$ and for all
$i=1,\dots ,n$.
Therefore, $\tilde c(\xi, x) = 0$, and the function
$\tilde c : H^*\to C$,
$\xi \mapsto \tilde c(\xi , x)\in C$,
is continuous with respect to the discrete topology on~$C$.

Let us fix
$a\in C$, $\xi\in H^*$ and denote by
$\tilde a(\xi)\in \End C$
the linear map that turns $b\in C$
into $ \tilde c(\xi,x)\in C$, where $\tilde c$ is constructed from
$a,b\in C$ as above.
Since $\tilde c$ is continuous with respect to the
discrete topology on $C$,
the map
$\alpha : \xi\mapsto a(\xi )$ is continuous with respect to
the finite topology on $\End C$.
By Lemma \ref{loc-cont}(2),
there exists a locally regular map
$a_1 : G\to \End C$ such that
$\tilde a_1(\xi )=\alpha (\xi)$ for every $\xi \in H^*$.

Note that for all $g,z\in G$, $b\in C$, $u\in M$ we have
\begin{multline}\nonumber
(a_1(g) b)(z)u = \sum\limits_{i\in I} h_i(g)(\alpha(\xi)b)(z)u
=
\sum\limits_{i\in I}h_i(g)(\langle \xi_i,\cdot\rangle\otimes \mathrm{ev}_z \otimes\idd)
 c(x,y)u  \\
 =
 (\mathrm{ev}_g\otimes \mathrm{ev}_z \otimes \idd)c(x,y)u = (a\oo{g} b)(z)u,
\end{multline}
since
$\sum\limits_{i\in I} h_i(g)\langle \xi_i,\cdot \rangle = \mathrm{ev}_g$.
Therefore,
 $a_1(g)b = (a\oo{g} b)$ and the function
$(a\oo{x} b)$ is regular.

Associativity of the conformal algebra
$\Cend^{G,V} M$ follows immediately from the definition
of operations \eqref{eq:G-prod}.
\end{proof}

The most interesting case is when $M$ is a free
$A$-module. Then one may identify $M$
with the space of regular vector-valued functions
on $V$,
and $\Cend^{G,V} M$ is a collection of
transformation rules of this space by means of the group $G$
 (e.g., the left shift $L$ belongs to $\Cend^{G,V} M$).

If $M$ is a free $n$-generated $A$-module then let us denote
$\Cend^{G,V} M$ by $\Cend^{G,V}_n$.
The structure of the conformal algebra
$\Cend^{G,V}_n $ is completely described by the following statement.

\begin{thm}
The conformal algebra $\Cend^{G,V}_n$ is isomorphic to
the conformal algebra
$\Diff(A\otimes \mathbb M_n(\Bbbk), \Delta_A\otimes \idd)$.
\end{thm}

\begin{proof}
Let us fix a basis
$\{e_k\}_{k=1}^n$ of a free $A$-module $M$.
Then $M$ is isomorphic to $A\otimes \Bbbk^n$.

An arbitrary element $a\in \Cend^{G,V}_n$ is uniquely defined
by a collection of regular $M$-valued functions
$a(x)e_k\in H\otimes M$,
$k=1,\dots, n$. Assume
\[
a(x)e_k = \sum\limits_{i\in I} h_i\otimes u_{ik}, \quad u_{ik} =
\sum\limits_{j\in J} f_j\otimes v_{ijk},
\]
where
$\{h_i\}_{i\in I}$ is a basis of $H$,
$\{f_j\}_{j\in J}$ is a basis of $A$,
$v_{ijk}\in \Bbbk^n$.
Then
\[
  a(z)\left (\sum\limits_{k=1}^{n}g_ke_k\right)
=
 \sum\limits_{i\in I,j\in J} \sum\limits_{k=1}^{n}
 h_i(z)f_j L_z g_k v_{ijk},
\quad
g_k\in A.
\]
Consider the linear maps
$a_{ij}\in \End \Bbbk^n\simeq \mathbb M_n(\Bbbk)$
that are defined by their values on the canonical basis:
$a_{ij}e_k = v_{ijk}$, $k=1,\dots ,n$.
Denote
\[
C=\Diff(A\otimes \mathbb M_n(\Bbbk), \Delta_A\otimes \idd)=H\otimes A \otimes
\mathbb M_n(\Bbbk ),
\]
and define
\[
\Phi: a\mapsto
  \Phi(a) = \sum\limits_{i\in I,j\in J} S(h_i) \otimes f_j\otimes a_{ij}\in C.
\]
The map $\Phi: \Cend^{G,V}_n \to C$ constructed is
$H$-linear, and for all $a,b\in \Cend^{G,V}_n$, $z\in G$ we have
\begin{multline}\nonumber
(a\oo{z} b)(x)e_k = a(z)b(xz^{-1})e_k =
\sum\limits_{i\in I,j\in J} h_{i(1)}h_{i(-2)}(z)\otimes a(z)(f_jb_{ij}e_k)\\
=
\sum\limits_{i,l\in I,\, j,p\in J}
h_l(z)h_{i(1)}h_{i(-2)}(z) \otimes f_pL_zf_j \otimes a_{lp}b_{ij}e_k
=(\Phi(a)\oo{z} \Phi(b))(x)e_k,
\end{multline}
where the right-hand side is computed by \eqref{eqGprod-tmp}.
Therefore, $\Phi(a\oo{z} b) =\Phi(a)\oo{z} \Phi(b)$, $z\in G$,
so $\Phi $ is a homomorphism of conformal algebras.
The inverse map $\Phi^{-1}$ is given by the rule
\[
  \sum\limits_{i\in I,j\in J} h_i\otimes f_j\otimes a_{ij}
  \mapsto a\in \Cend^{G,V}_n,
\]
where
$a(z) = \sum\limits_{i\in I,j\in J}
  h_i(z^{-1})f_jL_z\otimes a_{ij}\in \End M$.
Hence, $\Phi $ is an isomorphism.
\end{proof}

Hereinafter, we identify
 $\Cend^{G,V}_n$ and the conformal algebra
$H\otimes A\otimes \mathbb M_n(\Bbbk)
     \simeq H\otimes \mathbb M_n(A)$
with operations \eqref{eqGprod-tmp}.

Define an $H$-linear map
$\mathcal F : H\otimes A \to H\otimes A$
as follows:
$\mathcal F(f\otimes a) = fa_{(-1)}\otimes a_{(2)}$.
This map is invertible:
$\mathcal F^{-1}(h\otimes a) = ha_{(1)}\otimes a_{(2)} $.
We will also denote by
 $\mathcal F$
the linear transformation $\mathcal F\otimes \idd$
of the space
$ H\otimes A\otimes \mathbb M_n(\Bbbk)
\simeq H\otimes \mathbb M_n(A)$.

\begin{prop}\label{prop:Ideals}
{\rm 1.}
A right ideal of
$\Cend^{G,V}_n$ is of the form
$B=H\otimes B_0$, where $B_0$ is a right ideal of
$\mathbb M_n(A)$. \\
{\rm 2.}
A left ideal of
$\Cend^{G,V}_n$ is of the form
$B= \mathcal F(H\otimes B_0)$,
where $B_0$ is a left ideal of
$\mathbb M_n(A)$,\\
{\rm 3.}
Conformal algebra $\Cend^{G,V}_n$
is simple if and only if $G$ acts transitively on~$V$.
\end{prop}

\begin{proof}
Statements 1 and 2 can be proved in a routine way
in accordance with a scheme from
\cite{BKL}.

To prove the statement~3, assume $B$ is a two-sided
ideal of $\Cend^{G,V}_n$.
Then $B=H\otimes B_0$ as a right ideal, and $B_0$
has to be a two-sided ideal invariant with respect to the left-shift action
of~$G$. This implies $B_0=A_0\otimes \mathbb M_n(\Bbbk )$,
where $A_0$ is a $G$-invariant ideal of $A=\Bbbk[V]$.
The converse is also true: if $A_0$ is a $G$-invariant ideal of $A$
then $B=H\otimes A_0\otimes \mathbb M_n(\Bbbk)$
is a (two-sided) ideal of
$\Cend^{G,V}_n$.
It remains to note that
the coordinate algebra of a $G$-set $V$
contains no non-trivial $G$-invariant ideals if and only if
$G$ acts transitively on~$V$.
\end{proof}

\subsection{Algebra of operators}
Consider the case when
 $V=G$  and  $G$ acts on $V$ by left multiplications.
Then $A=H$, $\Delta_A =\Delta $.
Let us denote by $\Cend_n$ the conformal algebra
$\Cend^{G,G}_n$.
The main investigation tool of the conformal algebra
$\Cend_n$ is the (ordinary) algebra generated by
linear operators of the form
$a(g)\in \End M_n$, $a\in \Cend_n$, $g\in G$.
Throughout the rest of the paper, the field
$\Bbbk $ is supposed to be algebraically closed.

As above,
we identify a free $n$-generated $H$-module
$M_n$ with $H\otimes \Bbbk^n$ choosing a basis.

Denote by $W_n$ the linear span in $\End M_n$ of all operators
of the form
$a(g)$, $a\in \Cend_n$, $g\in G$.
This is an (ordinary) associative subalgebra of $\End M_n$
since \eqref{eq:G-prod} implies
\begin{equation}\label{eq:prod}
a(g)b(z) = (a\oo{g} b)(zg), \quad g_1,g_2\in G.
\end{equation}
Algebra $W_n\subseteq \End M_n$ is equipped by the induced finite topology.

\begin{exmp}
It is easy to see that if
$a=\sum\limits_{i,j\in I}h_i\otimes h_j\otimes a_{ij}\in \Cend_n$
then
$a(g) = \sum\limits_{i,j\in I}h_i(g^{-1})h_jL_g\otimes a_{ij}$.
Therefore,

1) if $G=\{e\}$ then $W_n = \End \Bbbk^n\simeq \mathbb M_n(\Bbbk)$;

2) if  $G=\mathbb A^1$, $\mathrm{char}\,\Bbbk=0$,
then the algebra $W_n$
consists of operators
$$
  \sum\limits_{i=1}^m A_i(x)e^{\lambda_i\partial},
\quad A_i\in \mathbb M_n(\Bbbk[x]), \quad \lambda_i\in \Bbbk,
\ m\ge 0,
$$
that act on
$\Bbbk[x]\otimes \Bbbk^n$,
where $\partial $ is the ordinary derivation with respect to~$x$.
\end{exmp}

\begin{lem}\label{CH2-lemUH}
Suppose  $g_1,\dots, g_m$ be pairwise different elements of~$G$,
$m\ge 1$,
and let $U\subseteq G$ be a nonempty Zariski open subset.
Then there exist
$f_1,\dots, f_m \in H$ such that
\begin{equation}\label{eq:det}
\left| \begin{matrix}
L_{g_1}f_1 & L_{g_1}f_2 & \dots &L_{g_1}f_m \\
L_{g_2}f_1 & L_{g_2}f_2 & \dots &L_{g_2}f_m \\
\hdotsfor4 \\
L_{g_m}f_1 & L_{g_m}f_2 & \dots & L_{g_m}f_m
\end{matrix}\right |(z)\ne 0
\end{equation}
for some $z\in U$.
\end{lem}

\begin{proof}
For $m=1$ the statement is obvious.
Assume the lemma is true for a collection of $m-1$ pairwise different
$g_2,\dots, g_m\in G$.
Denote by
$h$ the determinant
$|L_{g_i}f_j|_{i,j=2,\dots,m}\in H$.
Then
$U'=\{z\in U\mid h(z)\ne 0\}$ is a nonempty open subset
of~$G$.

Consider an element $g_1\in G$ different from $g_2,\dots, g_m$;
then for any $z\in U'$ the elements
$z_i=g_i^{-1}z$, $i=1,\dots,m$,
are pairwise different.
Choose
$f'_1\in \{f\in H\mid f(z_2)=\dots=f(z_m)=0,\, f(z_1)\ne 0\}$.
Such a function $f_1'$ exists since
$z_1$ does not belong to the closure of
$\{z_2,\dots, z_m\}\subset G$.
Then $f_1=L_{g_1^{-1}}f'_1$ satisfies~\eqref{eq:det}.
\end{proof}

The following statement leads to a
``normal form'' of elements of the algebra $W_n$.

\begin{lem}\label{lem:W}
{\rm 1.} For every $a\in \Cend_n$, $g\in G$
the equality
$(a\oo{g} \Cend_n)=0$ holds if and only if
$a(g)=0$.

{\rm 2.}
If $\sum\limits_i a_i(g_i)=0$, where
$g_i\in G$ are pairwise different,
then  $a_i(g_i)=0$ for all~$i$.
\end{lem}

\begin{proof}
Statement 1 follows immediately from \eqref{eq:prod}.

To prove 2, suppose
$\sum\limits_{i=1}^m a_i(g_i)=0$.
Then for all $f\in H$
\[
0=\sum\limits_{i=1}^m a_i(g_i)f = \sum_i L_{g_i}f a(g_i)
\]
by \eqref{TinvCend}.
Assume there exists $u\in M_n$ such that $a(g_k)u \ne 0$ for some~$k$.
Fix such an index~$k$ and denote by $U$ the set of all
$z\in G$ such that
$(a(g_k)u)(z)\ne 0$
(we consider $M_n$ as the module of regular $\Bbbk^n$-valued functions on~$G$).
The set $U$ is Zariski open, so by Lemma \ref{CH2-lemUH}
there exist
$f_1,\dots, f_m\in H$
such that
$h(z)\ne 0$
for some $z\in U$,
where
$h=\det |L_{g_i}f_j|$.
But
\[
0=\sum\limits_{i=1}^m
(a_i(g_i)f_ju)(z)
=
\sum\limits_{i=1}^m (L_{g_i}f_j)(z) (a(g_i)u)(z),\quad j=1,\dots ,m,
\]
so $h(z)(a(g_i)u)(z)=0$ for all~$i$, therefore,
$(a(g_k)u)(z)=0$. The contradiction obtained proves~2.
\end{proof}

\begin{lem}\label{lem:Wmod}
{\rm 1.}
Algebra $W_n$
is a topological left $H$-module with respect to
the discrete topology on $H$ and finite topology on~$W_n$.

{\rm 2.}
Conformal algebra
$\Cend_n$ is a topological left $W_{n}$-module
with respect to the finite topology on $W_n$
and discrete topology on
$\Cend _n$.
\end{lem}

\begin{proof}
It follows from Lemma \ref{lem:W} that the actions of
$H$ on $W_n$ and of $W_n$ on $\Cend_n$ are well defined by
the rules
\begin{gather}
h\cdot a(g)=h(g^{-1})a(g),\quad a\in \Cend_n,\ h\in H,\ g\in G,
  \label{Wmod1}\\
a(z)\cdot b = (a\oo{z} b),\quad  a,b\in \Cend_n,\ z\in G.
  \label{Wmod2}
\end{gather}
The associativity condition is obvious for
\eqref{Wmod1} and could be easily checked for \eqref{Wmod2}.
It remains to show that these maps are continuous.

1. Suppose $W_n\ni \alpha_k \to 0$ as $k\to \infty$.
Each $\alpha_k$ can be presented in the normal form from
Lemma~\ref{lem:W}(2).
It is necessary to check that
$(h\cdot \alpha_k)\to 0$ for a fixed $h\in H$.

Indeed, for every
$a\in \Cend_n$, $h\in H$, $g\in G$
we have
\begin{multline}\nonumber
  h_{(1)}a(g)h_{(-2)} = h_{(1)}L_gh_{(-2)}a(g)
 =h_{(1)}h_{(-3)}(g)h_{(-2)}a(g) \\
 = \varepsilon (h_{(1)})h_{(-2)}(g)a(g)
 = h(g^{-1})a(g),
\end{multline}
therefore,
\[
h_{(1)} \alpha_k h_{(-2)} = h\cdot\alpha_k .
\]
Since $\Delta(h)=h_{(1)}\otimes h_{(2)}$
involves only a finite number of summands,
the sequence
$(h\cdot \alpha_k)$ converges to zero
in the sense of finite topology on~$W_n$.

2. Consider a sequence $\alpha _k\to 0$ and an element
$x\in \Cend_n$.
We need to show that
$\alpha_k\cdot x = 0$ in $\Cend_n$ for sufficiently large~$k$.
If
$x=1\otimes f \otimes a$,
$f\in H$, $a\in \mathbb M_n(\Bbbk)$,
then the claim obviously follow from the definition of finite topology
(the action of $\alpha _k$ on $x$ can be considered
on the columns of~$f\otimes a$ separately).
Note that
\[
\alpha_k \cdot hx = h_{(2)} ((h_{(-1)}\cdot \alpha_k)\cdot x),
\quad h\in H,
\]
therefore, statement~1 implies $\alpha _k\cdot hx = 0$
as $k\gg 0$.
\end{proof}

\begin{cor}\label{lem:conv}
{\rm 1.}
If $\{\alpha_k\}_{k\ge 0}$ is a fundamental sequence
if~$W_n$ then for every $a\in \Cend_n$
there exists a natural $m$ such that
$\alpha_p\cdot a =\alpha_{q}\cdot a$
for all $p,q\ge m$.

{\rm 2.}
If a sequence $\{\alpha _k\}_{k\ge 0}$
converges to $ \alpha \in W_n$ as $k\to \infty $
then $\alpha _k\cdot a = \alpha \cdot a$
for any $a\in \Cend_n$ as $k\gg 0$.
\qed
\end{cor}

\begin{prop}\label{prop:irred}
Algebra $W_n\subseteq \End M_n$
acts on $M_n$ irreducibly.
\end{prop}

\begin{proof}
Note that for every
$ f\in H\simeq M_1$, $f\ne 0$,
the submodule $W_1f$ is an ideal of $H$
invariant with respect to left shifts.
Hence, $W_1f=H$.

Consider an arbitrary element
$u=\sum\limits_{k=1}^{n}f_k\otimes e_k\in M_n$,
$u\ne 0$. Without loss of generality, assume
$f_1\ne 0$.
It is easy to see that for every $i=1,\dots , n$ and
for every $h\in H$ there exists
$\alpha_{i,h} \in W_n$ such that $\alpha_{i,h} u = h\otimes e_i$:
$\alpha_{i,h} =\sum\limits_{j} h_jL_{z_j}\otimes e_{i1}$,
where $\sum\limits_{j}h_jL_{z_j}f_1 = h$.
Therefore, for any
$v=\sum\limits_{k=1}^{n}h_k\otimes e_k\in M_n$
one may build
$\alpha =\sum\limits_{k=1}^{n} \alpha _{k,h_k}\in W_n$
such that $\alpha u=v$.
\end{proof}

Note that  $W_n$ contains the operators $\Gamma(h): M_n\to M_n$,
$\Gamma (h)u= hu$, $h\in H$, $u\in M_n$.
Indeed, $\Gamma (h) = hL_e = (1\otimes h\otimes E)(e)$,
where $E$ is the identity matrix.

\begin{prop}\label{prop:dense}
Let $S$ be a subalgebra of $W_n$ that acts irreducibly on
$M_n$, and suppose $W_n$ is closed under multiplication by operators
$\Gamma(h)$, $h\in H$,
on the left and on the right, i.e.,
$HS, SH \subseteq S$.
Then the space of $S$-invariant linear transformations of the
space~$M_n$ consists of $H$-module endomorphisms.
\end{prop}

\begin{proof}
Denote
\[
D=\End_SM_n :=\{\varphi\in \End M_n \mid \varphi\alpha=\alpha\varphi
\mbox{ for all } \alpha \in S\}.
\]
By the Schur Lemma, $D$ is a division algebra.
Consider arbitrary elements
$0\ne \alpha\in S$, $f\in H$, $\varphi \in D$.
Since $f\alpha, \alpha f \in S$, we have
\[
(\varphi f)\alpha =\varphi
(f\alpha)
 = (f \alpha)\varphi = (f\varphi)\alpha.
\]
Hence,
\begin{equation}\label{eq:comm}
[\varphi, f]\alpha = (\varphi f - f\varphi )\alpha =0.
\end{equation}
Note that
$\psi = \varphi f - f\varphi \in D$.
Indeed,
\[
\beta (\varphi f - f\varphi) = (\beta \varphi)f - (\beta f)\varphi
=  (\varphi \beta )f - \varphi(\beta f) = 0
\]
for every $\beta \in S$.
Relation \eqref{eq:comm} implies $[\varphi, f] = 0$,
and thus
$\varphi \in \End_H M_n$.
\end{proof}

\begin{lem}\label{lem:no_fields}
If $D$ is a division algebra in $\End_H M_n$ that contains
$\idd_{M_n}$ then  $D=\Bbbk $.
\end{lem}

\begin{proof}
The endomorphism algebra
 $\End_H M_n$
of the free $H$-module $M_n$ is isomorphic to the matrix algebra
$\mathbb M_n(H)$.
Suppose
$D\subseteq \mathbb M_n(H)$
is a division algebra, and let
$x$ stands for a nonzero element of~$D$ such that $x\notin \Bbbk E$,
where $E$ is the identity matrix.

Since $D$ contains $\lambda E$ for every $\lambda \in \Bbbk$,
the element $x+\lambda E\in D$ is nonzero, therefore, invertible
in $D$ for all $\lambda \in \Bbbk $.
Denote by $f(\cdot,\lambda)$ the function $z\mapsto \det(x +\lambda E)(z)$, $z\in G$.
This is a polynomial in $\lambda $ with coefficients in~$H$.
It is clear that
 $f(\cdot ,\lambda )=\lambda^n + \dots + \det x$,
$(\det x)(z)\ne 0$ at any point $z\in G$.
Since $\Bbbk $ is algebraically closed, for every $z\in G$
there exists $\lambda \in \Bbbk $ such that
$f(z,\lambda)=0$, which is impossible since
$x+\lambda E $ is invertible.
\end{proof}

\begin{thm}\label{thm:Wirred}
If $S\subseteq W_n$ is a subalgebra that acts irreducibly
on $M_n$ and $HS,SH\subset S$ then $S$
is a dense (over $\Bbbk$) subalgebra of $\End M_n$.
\end{thm}

\begin{proof}
Recall that by the Jacobson Density Theorem
(see, e.g., \cite{Jac}) an algebra $A$ with a faithful irreducible
module $M$ and with a centralizer $D = \End_A M$
is a dense subalgebra of
$\End_D M$ with respect to the finite topology.
The latter means that for an arbitrary linearly independent over $D$ finite set
$u_1,\dots ,u_m\in M$
and for every
$v_1,\dots, v_m\in M$, $m\ge 1 $,
there exists $\alpha\in A $ such that
$\alpha u_i =v_i$, $i=1,\dots, m$.

By Proposition~\ref{prop:dense}, the centralizer $D=\End_S M_n$
is embedded into $ \End_H M_n$.
It follows from Lemma~\ref{lem:no_fields} that $D=\Bbbk $.
Hence, $S$ is dense in $\End M_n$.
\end{proof}

\begin{cor}\label{cor:Wdense}
The algebra $W_n$ is dense in $\End M_n$. \qed
\end{cor}

\subsection{Automorphisms and irreducible subalgebras of the conformal algebra
$\Cend_n$}

\begin{thm}\label{thm:auto}
The group of automorphisms of the conformal algebra
$\Cend_n$
is isomorphic to the group of topological
(with respect to the finite topology)
$H$-invariant automorphisms of the algebra~$W_n$.
\end{thm}

\begin{proof}
Let $\Theta $ be an automorphism of the conformal algebra $\Cend_n$,
i.e., a bijective $H$-invariant map that preserves all operations
$(\cdot\oo{g}\cdot)$, $g\in G$.
Define $\theta : W_n\to W_n$ by the rule
\begin{equation}\label{eq:theta}
\theta (a(g)) = \Theta (a)(g),\quad g\in G, a\in \Cend_n.
\end{equation}
If $a(g)=0$ then $a\oo{g}\Cend_n=0$, therefore,
$\Theta(a)\oo{g} \Cend_n =0$ and $\theta(a(g))=\Theta(a)(g)=0$
by Lemma~\ref{lem:W}(1).
It follows from Lemma~\ref{lem:W}(2) that
$\theta $ is well-defined on $W_n$ by linearity.

This is straightforward to check that
$\theta $ is an $H$-invariant map:
\[
 \theta (h\cdot a(g)) = \theta (h(g^{-1})a(g)) = h(g^{-1})\Theta (a)(g)
=h\cdot (\Theta(a)(g)) = h\cdot \theta (a(g)).
\]

To prove the continuity of $\theta $, it is enough to show that for
an arbitrary converging sequence
$\alpha _k\to \alpha \in W_n$ its image $\{\theta (\alpha _k)\}_{k\ge 0}$
converges to  $\theta (\alpha )$ in~$W_n$.
Indeed, the definition of $\theta $ implies
\[
  \theta (\alpha _k)\cdot \Theta(a) = \Theta(\alpha _k\cdot a),
\quad   \theta (\alpha )\cdot \Theta(a) = \Theta(\alpha \cdot a),
\quad a\in \Cend_n,\ k\ge 0.
\]
Since $\Theta $ is surjective, $\theta (\alpha _k)\cdot a = \theta (\alpha) \cdot a$
for every $a\in \Cend_n$, $k\gg 0$, by Lemma~\ref{lem:Wmod}(2).
For the elements of the form
$a=(1\otimes f_j\otimes a_j)$, $f_j\in H$, $a_j\in \mathbb M_n(\Bbbk)$,
one may compute
\[
 (\theta (\alpha _k)\cdot a)(z)(1\otimes v)
 = \theta (\alpha _k)(f_j\otimes a_jv) = 0,
\quad v\in \Bbbk^n,               \ k\gg 0,
\]
therefore, the sequence $\theta (\alpha _k)$ converges to $\theta (\alpha )$.

The map $\theta $ has an inverse one that can be easily constructed
from $\Theta^{-1}$.
Hence, $\theta $ is bijective, and its inverse is also continuous.

It is easy to check that all products are preserved:
\begin{multline}\label{eq:Tinv-check}
\theta (a(g)b(z)) = \theta ((a\oo{g} b)(zg))
=\Theta(a\oo{g} b)(zg) = (\Theta(a)\oo{g}\Theta(b))(zg) \\
=\Theta(a)(g)\Theta(b)(z)=\theta (a(g))\theta (b(z)).
\end{multline}

Therefore, an arbitrary automorphism $\Theta $ of the conformal algebra
$\Cend_n$ gives rise to a continuous $H$-invariant automorphism
$\theta $ of the algebra~$W_n$. It follows from \eqref{eq:theta}
that the map $\Theta\mapsto \theta $ is a group homomorphism.

Conversely, let $\theta $ be a continuous $H$-invariant automorphism of $W_n$.
Consider the map $\Theta: \Cend_n \to \Cend_n$
built by the following rule: for an arbitrary $a\in \Cend_n$
set
$\Theta(a): g\mapsto \theta (a(g))$, $g\in G$.

To complete the proof, it is enough to show that
$\Theta(a)\in \Cend_n$ and $\Theta $ is an automorphism of $\Cend_n$.
For every $h\in H$, $g\in G$ we have
\begin{multline} \nonumber
\Theta (a)(g) h =\theta (a(g))h  = h_{(2)}(h_{(-1)}\cdot \theta (a(g)))
=h_{(2)}\theta (h_{(-1)}\cdot a(g)) \\
= h_{(2)}h_{(1)}(g) \theta (a(g))= L_gh \Theta (a)(g).
\end{multline}
Hence, $\Theta(a)$ is T-invariant.
By Corollary~\ref{cor:Wdense}, the algebra $W_n$
is dense in $\End M_n$. Therefore, $\theta $
can be expanded, in particular, to the elements
$\tilde a(\xi )\in \End M_n$, $\xi\in H^*$.
By Lemma~\ref{loc-cont}(1), the map $\xi \mapsto \tilde a(\xi )$
is continuous, hence, $\theta \tilde a$ is also continuous.
By Lemma~\ref{loc-cont}(2) there exists a locally regular map
$a_1: G\to \End M_n$ such that $\tilde a_1=\theta \tilde a$.
It is clear that $a_1 =\Theta(a)$, hence, $\Theta(a)\in \Cend_n$.

It follows from \eqref{eq:Tinv-check} that $\Theta $
preserves $(\cdot\oo{g}\cdot)$, $g\in G$,
The continuity of $\theta ^{-1}$ implies the existence of $\Theta^{-1}$
\end{proof}

If $C$ is a conformal subalgebra of $\Cend_n$
then by $W_n(C)$ we denote
$\Spann_{\Bbbk} \{a(g)\mid a\in C,\, g\in G\}\subseteq W_n$.
It is clear that $W_n(C)$ is a subalgebra of~$W_n$.

\begin{defn}\label{defn:irred}
A conformal subalgebra $C\subseteq \Cend_n$ is said to be {\em irreducible\/}
if $M_n$ contains no $W_n(C)$-invariant $H$-submodules except for
$\{0\}$ and $M_n$.
\end{defn}

Recall that a left (right) ideal of an algebra is essential
(see, e.g., \cite{Faith}) if it has a nonzero intersection with
every nonzero left (right) ideal of the algebra.

\begin{lem}\label{ess-irr}
A conformal subalgebra of the form
$\mathcal F(H\otimes B_0)$
of $\Cend_n$, where $B_0$ is a left ideal of $\mathbb M_n(H)$,
is irreducible if and only if
$B_0$ is essential.
\end{lem}

\begin{proof}
Suppose $B_0$ is an essential left ideal.
Then the Goldie theory implies
(see, e.g., \cite{Faith}) that $B_0$ contains a matrix $a$
which is not a zero divisor in $\mathbb M_n(H)$.
Let $a=\sum\limits_{i} f_i\otimes a_i$,
$f_i\in H$, $a_i\in \mathbb M_n(\Bbbk )$.
The element
$x=\sum\limits_{i}\mathcal F(1\otimes f_i\otimes a_i)$
belongs to $C=\mathcal F(H\otimes B_0)$, and it follows from \eqref{eq:prod}
that $W_nx(z)\subseteq W_n(C)$ for every $z\in G$.
If $0\ne u=\sum\limits_{k=1}^{n}h_k\otimes e_k\in M_n$ then
\[
  x(z)u = \sum\limits_{k=1}^{n} \sum\limits_{i} L_z(f_ih_k)\otimes a_ie_k,
\]
and for $z=e $ we obtain $x(e)u = a u\ne 0$
since $a$ is not a zero divisor.
Therefore, $W_n(C)u\supseteq W_n x(e)u = W_n au = M_n$
by Proposition~\ref{prop:irred}, i.e.,
$M_n$ does not contain even a $W_n(C)$-invariant subspace.

Conversely, suppose $C=\mathcal F(H\otimes B_0)$ is an irreducible subalgebra of
$\Cend_n$.
Since $H$ is a semiprime Noetherian commutative algebra,
the matrix algebra $\mathbb M_n(H)$ is semiprime left and right Noetherian.
In particular, by the Goldie theory
$\mathbb M_n(H)$ is left and right Ore, and its classical (left) quotient algebra
$Q=Q(\mathbb M_n(H))$
is semisimple Artinian.

Let $R$ stands for the set of all elements of $\mathbb M_n(H)$ invertible in
$Q=R^{-1}\mathbb M_n(H)$.
Assume that $B_0$ is not essential. Then
$R^{-1}B_0$ is a proper left ideal of $Q$
(else $B_0$ contains an element of $R$, therefore, $B_0$ is essential).
By the Wedderburn---Artin Theorem,
$Q$ is a direct sum of full matrix rings over division algebras.
Then $R^{-1}B_0$ has a nonzero right annihilator $\ann_r (R^{-1}B_0)$ in $Q$.
If $0\ne r^{-1}a\in \ann_r (R^{-1}B_0)$,
$r\in R$, $a\in \mathbb M_n(H)$,
then the right Ore condition implies
$r^{-1}a = bs^{-1}$,
$s\in R$, $0\ne b\in \mathbb M_n(H)$.
It is clear that $B_0b =0$, therefore, $B_0$
has a nonzero right annihilator in
$\mathbb M_n(H)$.

Denote by $M'$ the $H$-submodule of $M_n$
generated by all $u\in M_n$ such that $B_0u=0$.
We have shown $M'\ne 0$, but for every
$a\in B_0$, $h\in H$, $z\in G$, $u\in M'$
we have
\[
 \mathcal F(h\otimes a) (z)u
= (ha_{(-1)})(z^{-1}) a_{(2)}L_z u
= h(z^{-1})L_z(au) =0.
\]
Hence, $M'$ is a $W_n(C)$-invariant $H$-submodule.
This is a contradiction with irreducibility of~$C$.
\end{proof}

\begin{thm}\label{thm:irred}
Let $C$ be an irreducible subalgebra of $\Cend_n $.
Then $C_1=(1\otimes H\otimes E)C$
is a left ideal of~$\Cend_n$,
$C_1 = \mathcal F(H\otimes B_0)$,
where $B_0$ is an essential left ideal of~$\mathbb M_n(H)$.
\end{thm}

\begin{proof}
Let $S = W_n(C)$. Then
$S_1 = HS $ is also a subalgebra of $W_n$ since
$a(g)h  = L_gh a(g)$ for $a\in C$, $g\in G$, $h\in H$.
It is also clear that
$C_1=(1\otimes H\otimes E)C$ is a conformal subalgebra of $\Cend_n$,
and $S_1=W_n(C_1)$.

If $C$ is irreducible in the sense of Definition~\ref{defn:irred}
then $S_1=HS$ acts irreducibly on $M_n$.

By Theorem \ref{thm:Wirred}, $S_1$ is a dense subalgebra of
$\End M_n$. By Lemma~\ref{lem:Wmod}(2), the conformal algebra $C_1$
can be considered as a topological left $S_1$-module.

Suppose $\{\alpha_k\}_{k\ge 0}$ is a sequence in $S_1$
which converges to $\alpha \in W_n$ as $k\to \infty $.
Then, by Corollary~\ref{lem:conv}, for every
$a\in \Cend_n$
there exists a natural $m$ such that
$\alpha_k\cdot a =\alpha\cdot a$ for all $k\ge m$.

Hence, for all $a\in C_1$, $\alpha =b(g)$, $b\in \Cend_n$, $g\in G$,
we have $C_1\ni \alpha _k\cdot a = \alpha \cdot a = (b\oo{g} a)$.
Therefore, $C_1$ is a left ideal of $\Cend_n$.
Proposition~\ref{prop:Ideals} implies
$C_1= \mathcal F(H\otimes B_0)$,
where $B_0$ is a left ideal of $\mathbb M_n(H)$
which is essential by Lemma~\ref{ess-irr}.
\end{proof}

If $G=\{e\}$ then Theorem \ref{thm:irred} turns into the Burnside Theorem.
If $G=\mathbb A_1$ then essential ideals of $\mathbb M_n(H)$
are of the form $\mathbb M_n(H)Q$, $\det Q\ne 0$.
The following statement is a corollary of Theorem~\ref{thm:irred}.

\begin{thm}[\cite{BKL, Ko1}]
Let $C\subseteq \Cend_n$ is an irreducible conformal subalgebra
(over $\mathbb A^1$, $\mathrm{char}\,\Bbbk =0$). Then either
$C=\Cend_{n,Q}:=\mathcal F(H\otimes \mathbb M_n(H)Q)$,
$\det Q\ne 0$,
or
$C=(1\otimes P^{-1})(H\otimes 1\otimes \mathbb M_n(\Bbbk))
(\mathcal F(1\otimes P))$,
where
$P\in \mathbb M_n(H)$, $\det P\in \Bbbk\setminus\{0\}$.
\end{thm}

Theorem \ref{thm:irred} is valid not only for connected groups.
Let us consider the case of ``intermediate complexity'' between
trivial group and affine line: when $G$ is a finite group.
In this case
$H=\Bbbk[G]\simeq (\Bbbk G)^*$, where $\Bbbk G$ is the group algebra of
$G$ considered as a Hopf algebra in the ordinary way.

\begin{exmp}
Let $G_1$ be a subgroup of a finite group $G$, $V=G/G_1$.
Then $C_{G_1}=\Cend_n^{(G,V)}$ is an irreducible conformal subalgebra
of $\Cend^{(G,G)}_n$.
\end{exmp}

\begin{thm}\label{thm:finitecase}
Let $C$ be an irreducible conformal subalgebra of
$\Cend_n^{(G,G)}$, $|G|<\infty $.
Then there exists an automorphism
$\theta $  of the conformal algebra $\Cend_n$
such that
$\theta(C)=C_{G_1,\chi }$,
\[
C_{G_1,\chi} =\left\{
\sum\limits_{g\in G}T_g\otimes \sum\limits_{k=1}^p\sum\limits_{\alpha\in G_k}
T_\alpha \otimes \chi (g,\alpha) a_{g,k} \mid a_{g,k}\in \mathbb M_n(\Bbbk )
\right\},
\]
where $G_1$ is a subgroup of $G$,
$\{G_1,\dots, G_p\}=\{gG_1\mid g\in G\}$,
and
$\chi: G\times G \to \Bbbk ^*$
is a function satisfying the following condition:
for every
$g,h\in G$, $k\in \{1,\dots, p\}$
the value
\[
\frac{\chi (g,\gamma )\chi(h,g^{-1}\gamma )}{\chi(gh, \gamma)} \in \Bbbk^*
\]
does not depend on the choice of a representative
$\gamma \in G_k$.
\end{thm}

For example, if $\chi \equiv 1$ then the conformal algebra $\Cend_n^{G,G/G_1}$
is isomorphic to $C_{G_1,\chi }$.

\begin{proof}
Sine $H=\Spann\{T_g\mid g\in G\}$,
$T_g (\gamma) = \delta_{g,\gamma}$,
then
$T_gT_\gamma = \delta_{g,\gamma}T_g$,
and every conformal subalgebra $C$ of
$\Cend_n =H\otimes H\otimes \mathbb M_n(\Bbbk )$
can be presented as
$C=\bigoplus\limits_{g\in G}\Bbbk T_g\otimes S_g$,
where
$S_g =\{x\in H\otimes \mathbb M_n(\Bbbk)\mid T_g\otimes x \in C \}$.

It is easy to see that $S_e$ is a subalgebra of
$H\otimes \mathbb M_n(\Bbbk )$.
Moreover,
\begin{equation}\label{eq:graded}
 S_g (L_{g^{-1}} S_h) \subseteq S_{gh }
\end{equation}
for all $g,h \in G$.

Let us identify
$H\otimes \mathbb M_n(\Bbbk )$
with
$\bigoplus\limits_{\gamma \in G} \mathbb M_n(\Bbbk )$
and denote by $\pi_g$, $g\in G$,
the canonical projections
$S_e \to \mathbb M_n(\Bbbk )$.

If $C$ is irreducible then Theorem \ref{thm:irred} implies
\begin{equation}               \label{eq:density}
\sum\limits_{\gamma \in G}(T_\gamma \otimes 1)S_g = H\otimes \mathbb M_n(\Bbbk)
\end{equation}
for each $g\in G$.
Thus, $S_e$ is a subdirect sum of matrix algebras, hence,
$S=I_1\oplus \dots \oplus I_p$, where
$I_k\simeq \mathbb M_n(\Bbbk )$ are two-sided ideals of~$S_e$.
Denote
$G_k = \{g\in G \mid \pi_g(I_k)\ne 0 \}\subseteq G$, $k=1,\dots, p$.
Since $I_kI_l = 0$ for $k\ne l$, we have
$G_k\cap G_l =\emptyset$.
It follows from \eqref{eq:density} that
$G=\bigcup\limits_{1\le k\le p} G_k$.
Let us enumerate the sets $G_k$ in such a way that
$e\in G_1$.

The maps $\pi_g: I_k\to \mathbb M_n(\Bbbk )$,
$g\in G_k$, are isomorphisms. Therefore,
\begin{equation}\label{eq:Se}
S_e=\left\{
\sum\limits_{k=1}^p \sum\limits_{g\in G_k} T_g\otimes a_k^{\theta_g}
\mid
a_k\in \mathbb M_n(\Bbbk)\right\},
\end{equation}
where $\theta_g$, $g\in G$, are automorphisms of the algebra
$\mathbb M_n(\Bbbk )$.

Relation \eqref{eq:density} also implies that
for every $g,\gamma \in G$, $a\in \mathbb M_n(\Bbbk )$
there exists $x\in S_g$ such that
$T_\gamma x:=(T_\gamma \otimes 1)x= T_\gamma\otimes a$.

\begin{lem}
The set $G_1$ is a subgroup of $G$, and
$G/G_1=\{gG_1\mid g\in G\}=\{G_k\mid k=1,\dots, p \}$.
\end{lem}

\begin{proof}
Assume there exist $k,m\in \{1,\dots, p\}$, $g\in G$
such that
$gG_k\cap G_m\ne \emptyset, G_m$.
Consider an arbitrary
$x=\sum\limits_{\gamma\in G} T_\gamma \otimes a_\gamma\in S_g$,
$a_{\gamma_0}\ne 0$ for some $\gamma_0\in gG_k\cap G_m$.
Then by
\eqref{eq:graded} and \eqref{eq:Se}, we have
\[
y_1=\sum\limits_{\gamma\in gG_k} T_\gamma \otimes a_\gamma \in S_g,
\quad
y_2 = \sum\limits_{\gamma\in gG_k\cap G_m}
T_\gamma \otimes a_\gamma \in S_g.
\]
Let us choose
$z \in S_{g^{-1}}$ in such a way that
$T_{\gamma_0}z=T_{\gamma_0}\otimes E$
and
$w=y_2L_{g^{-1}}z\in S_e$.
The element $w$ has the following properties:
$T_{\gamma_0}w = a_{\gamma_0}\ne 0$, $T_\gamma w =0$ for
$\gamma \in G_m\setminus gG_k$. This is a contradiction to~\eqref{eq:Se}.

Therefore, left multiplication by
$g\in G$ permutes the sets $\{G_1,\dots, G_p\}$.
\end{proof}

\begin{lem}
{\rm 1.}
Let $\sigma_{g,\alpha}\in \End \mathbb M_n(\Bbbk )$,
$g,\alpha \in G$, be a collection of bijective linear
transformations.
Define an $H$-linear map $\sigma $ by the rule
\begin{equation}\label{eq:auto-s}
\begin{aligned}
\sigma :{}& \Cend_n\to \Cend_n, \\
& T_g\otimes T_\alpha \otimes a \mapsto T_g\otimes T_\alpha\otimes
   a^{\sigma_{g,\alpha}},
\quad g,\alpha\in G,\ a\in \mathbb M_n(\Bbbk).
\end{aligned}
\end{equation}
The map $\sigma $ is an automorphism of the conformal algebra
$\Cend_n$ if and only if
$(ab)^{\sigma_{gh,\alpha}}
= a^{\sigma_{g,\alpha}}  b^{\sigma_{h,g^{-1}\alpha}}$
for all $a,b\in \mathbb M_n(\Bbbk)$, $g,h,\alpha\in G$.

{\rm 2.} For an arbitrary set
$\theta_\alpha$, $\alpha\in G$,
of automorphisms of the algebra
$\mathbb M_n(\Bbbk )$ there exists an automorphism
$\sigma $ of the form \eqref{eq:auto-s}
such that
$\sigma _{e,\alpha}=\theta_\alpha $.
\end{lem}

\begin{proof}
To prove first statement, it is enough to check the condition
$\sigma(x\oo{g} y) =\sigma(x)\oo{g}\sigma(y)$,
$x,y\in \Cend_n$, $g\in G$.

To prove the second one,
consider matrices
$T_\alpha \in \mathbb M_n(\Bbbk )$
such that
$T_\alpha^{-1}aT_\alpha = a^{\theta_\alpha }$,
and define
$a^{\sigma_{g,\alpha}} = a^{\theta_{\alpha }} E_{g,\alpha}$,
where $E_{g,\alpha } = T_\alpha^{-1}T_{g^{-1}\alpha }$,
$g,\alpha \in G$, $a\in \mathbb M_n(\Bbbk )$.
It is easy to check that the maps
$\sigma _{g,\alpha }\in \End \mathbb M_n(\Bbbk )$
constructed satisfy the conditions of statement 1.
\end{proof}

Therefore, we may suppose that for the conformal subalgebra
$C\subseteq \Cend_n $ we have
$S_e = \bigoplus_{k=1}^p A_k\otimes \mathbb M_n(\Bbbk )$,
where $A_k = \Bbbk (\sum_{g\in G_k} T_g )\subset H$.
For others $g\in G$ the structure of the space $S_g$
can be clarified by means of \eqref{eq:graded} and \eqref{eq:density}.
Namely,
\[
 S_g = \left\{\sum\limits_{k=1}^p \sum\limits_{\gamma \in G_k}
   T_\gamma \otimes a_k^{\sigma_{g,\gamma}}  \mid a_k\in \mathbb M_n(\Bbbk )\right \},
\]
where $\sigma _{g,\gamma }$ are some bijective linear transformations of
$\mathbb M_n(\Bbbk )$.
To be more precise, fix a system of representative
 $g_k\in G_k$ and assume $\sigma_{g,g_k}= \idd$,
$k=1,\dots, p$.

Consider elements of the form
\[
x=\sum\limits_{k=1}^{p} \sum\limits_{\alpha\in G_k}
T_\alpha \otimes E^{\sigma_{g,\alpha}}\in S_g,
\quad
y = \sum\limits_{\beta\in G} T_\beta\otimes b\in S_e,
\]
where $b\in \mathbb M_n(\Bbbk )$ is an arbitrary matrix.
Comparing the expressions
\[
(T_e\otimes y)\oo{e} (T_g\otimes x) = T_g\otimes z,
\quad
(T_g\otimes x)\oo{g^{-1}} (T_e\otimes y) = T_g\otimes z',
\]
we may conclude that
$\pi_{g_k}(z) = \pi_{g_k}(z') =b$
for all $k=1,\dots, p$.
It is easy to derive
\[
a^{\sigma_{g,\gamma }} = \chi(g,\gamma) a, \quad \chi: G\times G \to \Bbbk^*,
\]
where $\chi (g,g_k) =1$.
Using the relation \eqref{eq:graded}
it is not difficult to obtain the following relations for
$\chi $:
\[
\chi(g,\alpha )\chi(h,g^{-1}\alpha) = \chi(h,g^{-1}g_k)\chi(gh, \alpha),
\quad \alpha\in G_k, \ g,h\in G.
\]
Since $C=\sum\limits_{g\in G}T_g\otimes S_g$,
we obtain $C=C_{G_1,\chi }$.
\end{proof}


\begin{thebibliography}{22}

\bibitem{K1}
Kac V.~G.
Vertex algebras for beginners. Second edition,
Univ. Lecture Series {bf 10}, AMS, Providence, RI, 1998.

\bibitem{Bor}
Borcherds R.~E.
Vertex algebras, Kac-Moody algebras, and the Monster,
Proc. Nat. Acad. Sci. U.S.A. {\bf 83} (1986) 3068--3071.

\bibitem{FLM}
Frenkel~I.~B., Lepowsky~J., Meurman~A.
Vertex operator algebras and the Monster,
Pure and Applied Math {\bf 134},
Academic Press, Boston, MA, 1998.

\bibitem{K3}
Kac V.~G.
Formal distribution algebras and conformal algebras,
XIIth International Congress in Mathematical Physics (ICMP'97),
Internat. Press, Cambridge, MA, 1999, 80--97.

\bibitem{BKL}
 Boyallian C., Kac~V.~G., Liberati J.~I.
 On the classification of subalgebras of $\Cend_N$ and $\mathrm{gc}_N$,
 J.~Algebra {\bf 260} (2003) no.~1, 32--63.

\bibitem{Ko1}
Kolesnikov~P.~S.
Associative conformal algebras with finite faithful representation,
Adv. Math. {\bf 202} (2006) no.~2, 602--637.

\bibitem{Ko4}
Kolesnikov P.~S.
Associative algebras related to conformal algebras,
Appl. Categ. Structures, to appear.

\bibitem{BDK}
Bakalov~B., D'Andrea~A., Kac~V.~G.
Theory of finite pseudoalgebras,
Adv. Math. {\bf 162} (2001) no.~1, 1--140.

\bibitem{GK}
Ginzburg~V., Kapranov~M.
Kozul duality for operads,
Duke Math. J. {\bf 76} (1994) no.~1, 203--272.

\bibitem{BD}
Beilinson A.~A., Drinfeld V.~G.
Chiral algebras,
Amer. Math. Soc. Colloquium Publications {\bf 51},
AMS, Providence, RI, 2004.

\bibitem{GKK}
Golenishcheva-Kutuzova M.~I., Kac V.~G.
$\Gamma$-conformal algebras,
J. Math. Phys. {\bf 39} (1998) no.~4, 2290--2305.

\bibitem{Re1}
Retakh~A.
Associative conformal algebras of linear growth,
J. Algebra {\bf 237} (2001) no.~2, 769--788.

\bibitem{Jac}
Jacobson N.
Structure of rings,
American Mathematical Society Colloquium Publications {\bf 37},
AMS, Providence, RI, 1956.

\bibitem{Faith}
Faith C.
Algebra: Rings, Modules and Categories I,
Springer-Verl., Berlin-Heidelberg-New York, 1973.


\end{thebibliography}
\end{document}